\documentclass[journal,twoside,web]{ieeecolor}
\usepackage{nicematrix}
\usepackage{generic}
\usepackage{amsmath,amssymb,amsfonts}
\usepackage{algorithmic}
\usepackage{graphicx}
\usepackage{textcomp}

\usepackage{array,graphicx,cite,xcolor}
\usepackage[thinlines]{easytable}
\usepackage{graphicx}
\usepackage{tikz}
\usetikzlibrary{shapes,calc,snakes,spy}
\usepackage{amsmath,amsfonts,amssymb,mathtools}

\usepackage{amsthm,mathrsfs}
\usepackage{sidecap}
\usepackage{dblfloatfix}
\usepackage{subcaption}

\usepackage[printonlyused]{acronym}
\acrodef{hac}[HAC]{hybrid angle control}
\acrodef{coi}[COI]{center-of-inertia}
\acrodef{ib}[IB]{infinite bus}
\acrodef{sg}[SG]{synchronous generators}
\acrodef{wrt}[w.r.t.]{with respect to}
\acrodef{agas}[AGAS]{almost global asymptotic stability}
\acrodef{lhs}[LHS]{left-hand side}  
\acrodef{rhs}[RHS]{right-hand side}  
\acrodef{rocof}[RoCoF]{rate of change of frequency}  
\acrodef{ilc}[ILC]{interlinking converter}  

\providecommand{\n}[1]{\lVert#1\rVert}
\providecommand{\ubar}[1]{\underline{#1}}
\providecommand{\tx}[1]{\text{\upshape{#1}}}	

\providecommand{\mc}[1]{\mathcal{#1}}   
\providecommand{\di}[1]{\mathrm{diag}\left(#1\right)}
\providecommand{\R}[1]{\mathbb{R}^{#1}}

\usepackage{xpatch}
\makeatletter
\AtBeginDocument{\xpatchcmd{\@thm}{\thm@headpunct{.}}{\thm@headpunct{}}{}{}}
\makeatother
\usepackage{arydshln}
\usepackage{multicol,multirow}
\usepackage{pgfplots}
\usetikzlibrary[pgfplots.colormaps,matrix]

\usetikzlibrary{shapes,arrows}
\providecommand{\n}[1]{\lVert#1\rVert}
\providecommand{\ubar}[1]{\underline{#1}}
\providecommand{\tx}[1]{\text{\upshape{#1}}}	

\providecommand{\mc}[1]{\mathcal{#1}}                                                  		

\newtheoremstyle{break}
{\topsep}{\topsep}%
{\itshape}{}%
{\bfseries}{}%
{\newline}{}%
\theoremstyle{break}

\newtheorem{remark}{Remark}
\newtheorem{assumption}{Assumption}

\newtheorem{theorem}{Theorem}
\newtheorem{corollary}{Corollary}

\usepackage[font={small}]{caption}

\usepackage[printonlyused]{acronym}
\acrodef{hac}[HAC]{hybrid angle control}
\acrodef{coi}[COI]{center-of-inertia}
\acrodef{ib}[IB]{infinite bus}
\acrodef{sg}[SG]{synchronous generators}
\acrodef{wrt}[w.r.t.]{with respect to}
\acrodef{agas}[AGAS]{almost global asymptotic stability}
\acrodef{lhs}[LHS]{left-hand side}  
\acrodef{rhs}[RHS]{right-hand side}  
\acrodef{rocof}[RoCoF]{rate of change of frequency}  
\acrodef{ilc}[ILC]{interlinking converter}  
\usepackage{cite}
\renewcommand{\baselinestretch}{0.9945}
\begin{document}
\title{Hybrid Angle Control and Almost Global Stability of Non-synchronous Hybrid AC/DC Power Grids}
\author{Ali Tayyebi and Florian Dörfler
\thanks{\scriptsize Ali Tayyebi is with the Automatic Control Laboratory, ETH Z\"{u}rich, 8092 Z\"{u}rich, Switzerland and Hitachi Energy Research (HER), 72226 V\"{a}sterås, Sweden,  e-mail: ali.tayyebi@hitachienergy.com. Florian D\"{o}rfler is with the Automatic Control Laboratory, ETH Z\"{u}rich, 8092 Z\"{u}rich, Switzerland, e-mail: dorfler@ethz.ch. This work was funded by the Austrian Institute for Technology, HER, and ETH Z\"{u}rich.}}
\maketitle
\begin{abstract}
This paper explores the stability of non-synchronous hybrid ac/dc power grids under the grid-forming hybrid angle control strategy. We formulate dynamical models for the ac grids and transmission lines, interlinking converters, and dc generations and interconnections. Next, we establish the existence and uniqueness of the closed-loop equilibria for the overall system. Subsequently, we demonstrate global attractivity of the equilibria, local asymptotic stability of the desired equilibrium point, and instability and zero-Lebesgue-measure region of attraction for other equilibria. The theoretic results are derived under mild, parametric, and unified stability/instability conditions. Finally, relying on the intermediate results, we conclude the almost global asymptotic stability of the hybrid ac/dc power grids with interlinking converters that are equipped with hybrid angle control. Last, we present several remarks on the practical and theoretical aspects of the problem under investigation.   
\end{abstract}

\section{Introduction}\label{sec:introduction}
The global paradigm shift toward harvesting energy from renewable sources has recently led to the emergence of hybrid ac/dc power grids. Such systems are typically comprised of several non-synchronous ac power grids that interact with each other through dc/ac \acp{ilc} that are interconnected by a dc transmission network \cite{misyris2021zero,gross2021dualport,SG21}. For instance, Figure \ref{fig:HVDC_EU} illustrates an abstraction of the meshed hybrid ac/dc grids that have been recently evolving in Europe. 

The complex nonlinear dynamics of the hybrid ac/dc power grids with multiple timescales and interactions between the dc network, renewable generations, and ac girds renders the control of interlinking converters a daunting task. It has been recently reported that the grid-forming converter control techniques \cite{TGAKD20,CTGAKF19} are viable candidates for controlling the \acp{ilc} in hybrid ac/dc power grids \cite{WL20}. In particular, \cite{WL20} suggests that the matching control \cite{AJD18} exhibits superior dynamic performance in hybrid ac/dc grids compared to classic control schemes for the interlinking converters, e.g., dual-droop control among others \cite{ahrabi2020hybrid,peyghami2017autonomous}. Inspired by this intriguing observation, this work explores the stability certificates of the \ac{hac} \cite{HACCDC,HACTAC,HACPSCC} for multiple \acp{ilc}.

We provide dynamical models for the ac girds and transmission lines, \ac{ilc}s, dc generations and interconnections. Next, we formally prove the existence and uniqueness of equilibria for the closed-loop dynamics under a verifiable assumption. Further, a constructive analysis is presented that proves the \ac{agas} of hybrid ac/dc power grids with ILCs under the \ac{hac}.  
\begin{figure}[t!]
	\centering
	{\includegraphics[trim=1.5mm 1.5mm 1.5mm 1.5mm, clip,width=\columnwidth]{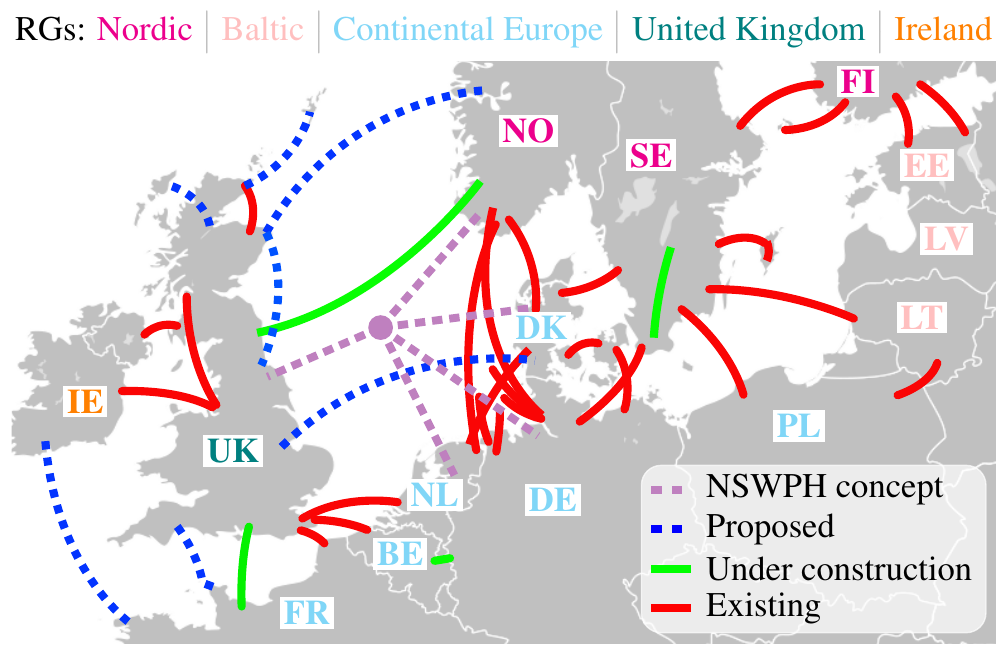}}
	\caption{The overview of the high voltage dc (HVDC) links and North Sea wind power hub (NSWPH) concept that connect the regional groups (RGs) in the Northern Europe and Baltic regions \cite{misyris2021zero}.}
	\label{fig:HVDC_EU}
\end{figure} 
\begin{figure*}[t!]
	\centering
	\includegraphics[width=\textwidth]{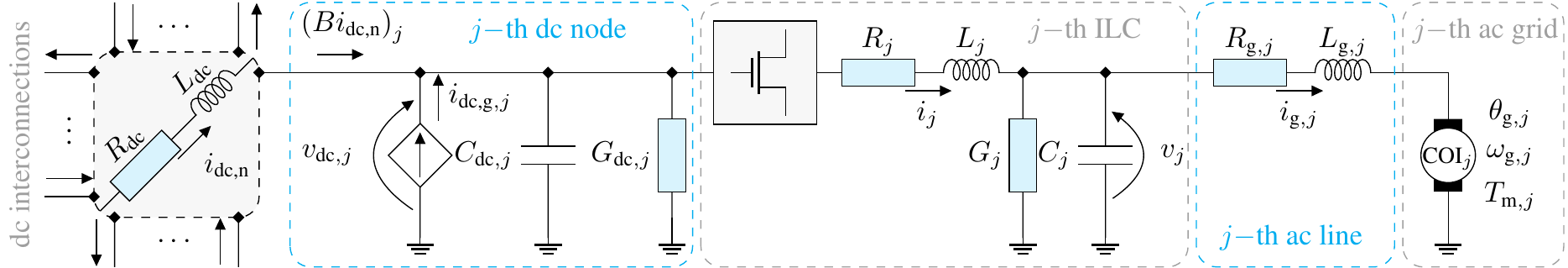}
	\caption{The hybrid ac/dc grid model comprised of the dynamic dc interconnections, $j-$th dc node, ILC, transmission line, and ac grid.}\label{fig:acdcgrid}
\end{figure*} 
\section{Hybrid AC/DC Grid Model Description}\label{sec:modelling}
In this section, we describe the dynamical model of the hybrid ac/dc grids. To begin with, consider $n\in\mathbb{Z}_{>0}$ ac grids that implies the inclusion of $n$ \acp{ilc} and define $\mathcal{N}_\tx{ac} \triangleq \{1,\ldots,n\}$ that collects the labels of the ac systems. Further, consider that the \acp{ilc} are interconnected via $m\in\mathbb{Z}_{>0}$ dc transmission lines; see Figure \ref{fig:acdcgrid} for the model configuration.
\subsection{Dynamic non-synchronous AC grids}
We model the ac grids by widely recognized aggregated dynamic \ac{coi} models \cite{SP98,MDMV16,HACTAC}, i.e.,
\begin{subequations}\label{eqs:coi}
	\begin{align}
	{ \dot{\theta}_\tx{g} } & = \omega_\tx{g},
	\label{eq:coi1}
	\\
	{ \dot{\omega}_\tx{g} } & = J^{-1}\left(T_\tx{m} - D_\tx{f} \omega_\tx{g} - D_\tx{d} ( \omega_\tx{g} - \omega_\tx{r} ) - T_\tx{e}\right),
	\label{eq:coi2}
	\\
	{ \dot{T}_\tx{m} } & = \tau_\tx{g}^{-1} \left(T_\tx{r} - \kappa_\tx{g} ( \omega_\tx{g} - \omega_\tx{r} ) - T_\tx{m}\right),
	\label{eq:coi3}
	\end{align}
\end{subequations}
where $\theta_\tx{g}  \triangleq  \left(  \theta_{\tx{g},1} , \ldots , \theta_{\tx{g},n} \right) \in \mathbb{S}^n$ denotes the stacked vector of the absolute phase angles of the ac grids, 
$\omega_\tx{g}  \triangleq  ( \omega_{\tx{g},1} , \ldots , \omega_{\tx{g},n} ) \in \R{n}$ denotes the vector of angular frequencies, 
$J  \triangleq  \di{ \{J_j\}^n_{j=1}} \in \R{ n \times n }_{>0}$ denotes the diagonal matrix of the moment of inertia constants, 
$T_\tx{m}  \triangleq  ( T_{\tx{m},1} , \ldots , T_{\tx{m},n} ) \in \R{n}$ denotes the vector of mechanical torques, 
$D_\tx{f}  \triangleq  \di{ \{D_{\tx{f},j}\}^n_{j=1}} \in \R{ n \times n }_{>0}$ denotes the diagonal matrix of the aggregated damping constants associated with the friction torques that are proportional to the absolute frequencies, 
$D_\tx{d}  \triangleq  \di{ \{D_{\tx{d},j}\}^n_{j=1}} \in \R{ n \times n }_{>0}$ denotes the diagonal matrix of the aggregated damping constants associated with the damper windings that are proportional to the frequency deviations, 
and $T_\tx{e}  \triangleq  ( T_{\tx{e},1} , \ldots , T_{\tx{e},n} ) \in \R{n}$ denotes the vector of electrical torques. 
Further, $\tau_\tx{g}  \triangleq  \di{ \{\tau_{\tx{g},j}\}^n_{j=1} } \in \R{ n \times n }_{>0}$ captures the aggregated turbine time constants, $T_\tx{r}  \triangleq  ( T_{\tx{r},1} , \ldots , T_{\tx{r},n} ) \in \R{n}$ denotes the reference mechanical torque inputs for the turbines, $\kappa_\tx{g}  \triangleq  \di{ \{\kappa_{\tx{g},j}\}^n_{j=1}}$ is the diagonal matrix of governor proportional control gains, and finally $\omega_\tx{r}  \triangleq  ( \omega_{\tx{r},1} , \ldots , \omega_{\tx{r},n} ) \in \R{n}_{>0}$ denotes the nominal frequencies of the ac grids. 

Our distinction for the damping terms in \eqref{eq:coi2} is inspired by the improved swing equation proposed in \cite{MDMV16}. In addition, the damping terms in \eqref{eq:coi2} can be seen as a representation of the primary and fast frequency controls that are respectively associated with the underlying synchronous machines and power converters in the aggregated COI model.

 Without loss of generality we assume that the COIs inject power into the dc interconnections, thus, $T_\tx{e}$ enters \eqref{eq:coi2} with a negative sign. However, since $T_\tx{e}$ can admit negative values, the sign convention in \eqref{eq:coi2} is not restrictive. Thus, it allows for modeling the power absorption by the ac grids, as well. Further, we do not assume identical nominal frequencies, thus, enabling \eqref{eqs:coi} to represent non-synchronous ac grids as in Figure \ref{fig:HVDC_EU}.  In Subsection \ref{subsec:equilibria}, we will characterize the connection of \eqref{eqs:coi} to the other system dynamics. In the sequel ---all three-phase quantities are transformed to rectangular dq-coordinates aligned with $\theta_\tx{g}$ in \eqref{eq:coi1}---, hence, the ac impedance and admittance matrices are dynamic and depend on $\omega_\tx{g}$. 

\subsection{Dynamic AC transmission lines}
The transmission lines that couple the ac grid models \eqref{eqs:coi} to the \acp{ilc}' ac-sides (see Figure \ref{fig:acdcgrid}) are modeled by \cite{HACTAC}
\begin{equation}
{\dot{i}_\tx{g}} = L_\tx{g}^{-1}\left(v - ( R_\tx{g} - \ubar{L}_\tx{g} \omega_\tx{g} \otimes J_2 ) i_\tx{g} - v_\tx{g}\right),
\label{eq:ac lines}
\end{equation}
where $i_\tx{g}  \triangleq  ( i_{\tx{g},1} , \ldots , i_{\tx{g},n} ) \in \R{2n}$ denotes line currents in the respective dq-frames that are aligned with the \ac{coi} angles $\theta_\tx{g}$ and $L_\tx{g}  \triangleq  \di{ \{ L_{\tx{g},j} \otimes I_2 \}^n_{j=1} } \in \R{ 2n \times 2n }$ denotes the augmented inductance matrix associated with transmission lines, and $\otimes$, $I_2$, and $J_2$ denotes the Kronecker product, 2-D identity matrix and rotation by $\pi/2$, respectively. %
Further, $v  \triangleq  ( v_1 , \ldots , v_n ) \in \R{2n}$ denotes the ac output voltages of the \ac{ilc}s, %
$R_\tx{g}  \triangleq  \di{ \{ R_{\tx{g},j} \otimes I_2 \}_{ j = 1 } ^ n} \in \R{2n \times 2n}$ is the augmented diagonal resistance matrix of the line impedance, %
and $\ubar{L}_\tx{g}  \triangleq  \di{ \{ L_{\tx{g},j} \}_{ j = 1 } ^ n} \in \R{n \times n}$ is the $n$-D diagonal reduction of $L_\tx{g}$. %
Last,  $v_\tx{g} \triangleq (v_{\tx{g},1},\ldots,v_{\tx{g},n})\in\R{2n}$ denotes the dynamic grid voltages. Note that the shunt capacitances of the ac transmission lines can be merged with the \ac{ilc}s filter capacitance due to their parallel connection; see Figure \ref{fig:acdcgrid}.   

\subsection{Interlinking DC-AC converters}
The \ac{ilc}s dynamics in the dq-frames aligned with the \ac{coi} angles in \eqref{eq:coi1} are given by \cite{TGAKD20,WL20,peyghami2017autonomous,zonetti2015control}
\begin{subequations}\label{eqs:ILC}
	\begin{align}
	{ \dot{\theta}_\tx{c} } & = \omega_\tx{c},
	\label{eq:ILC1}
	\\
	{ \dot{i}_\tx{dc,g} } & = \tau_\tx{dc}^{-1} \left(i_\tx{dc,r} - \kappa_\tx{dc} ( v_\tx{dc} - v_\tx{dc,r} ) - i_\tx{dc,g}\right),
	\label{eq:ILC2}
	\\
	{ \dot{v}_\tx{dc} } & = C_\tx{dc}^{-1} \left(B i_\tx{dc,n} + i_\tx{dc,g} - G_\tx{dc} v_\tx{dc} - m(\delta)^\top i\right),
	\label{eq:ILC3}
	\\
	\dot{i} & = L^{-1} \left(m(\delta) v_\tx{dc} - ( R - \ubar{L} \omega_\tx{g} \otimes J_2) i - v\right),
	\label{eq:ILC4}
	\\
	\dot{v} & = C^{-1} \left(i - ( G - \ubar{C} \omega_\tx{g} \otimes J_2 ) v - i_\tx{g}\right),
	\label{eq:ILC5}
	\end{align}
\end{subequations}
where $\theta_\tx{c}  \triangleq  ( \theta_{\tx{c},1} , \ldots , \theta_{\tx{c},n} ) \in \mathbb{S}^n$ denotes the \ac{ilc}s modulation angles evolving on the $n$-D torus and $\omega_\tx{c}  \triangleq  ( \omega_{\tx{c},1} , \ldots , \omega_{\tx{c},n} ) \in \mathbb{R}^n$ denotes the converter frequency. In subsection \ref{subsec:hac}, we will define $\omega_\tx{c}$ according to the \ac{hac} law.  %

The time constants associated with the first-order dc generation models are denoted by $\tau_\tx{dc}  \triangleq  \di{ \{ \tau_{\tx{dc},j} \}^n_{j=1} } \in \R{ n \times n }$ and $i_\tx{dc,g}  \triangleq  ( i_{\tx{dc,g},1} , \ldots , i_{\tx{dc,g},n} ) \in \R{n}$ denotes the currents flowing out of the dc current sources that are collocated with the \ac{ilc}s dc-sides, $i_\tx{dc,r}  \triangleq  ( i_{\tx{dc,r},1} , \ldots , i_{\tx{dc,r},n} ) \in \R{n}$ denotes the reference currents for the dc sources, and $\kappa_\tx{dc}  \triangleq  \di{ \{ \kappa_{\tx{dc},j} \}^n_{j=1} } \in \R{ n \times n }$ denotes the matrix of proportional dc voltage control gains. We remark that the inclusion of dc generation models in \eqref{eq:ILC2} is not necessary for establishing the stability results presented in Section \ref{sec:stability}. However, it improves the fidelity of the hybrid ac/dc grid modeling. 

The dc-link capacitances are denoted by the diagonal matrix $C_\tx{dc}  \triangleq  \di{ \{ C_{\tx{dc},j} \}^{n}_{j=1} } \in \R{ n \times n }$, %
the signed incidence matrix associated with the directed graph of the dc interconnections is denoted by $B \in \R{ n \times m }$, $i_\tx{dc,n}  \triangleq  ( {i_{\tx{dc,n},1} , \ldots , i_{\tx{dc,n},m}} ) \in \R{m}$ collects the dc edge currents, %
and $G_\tx{dc}  \triangleq  \di{ \{ G_{\tx{dc},j} \}^{n}_{j=1} } \in \R{ n \times n }$ denotes the nodal dc conductances that models the \ac{ilc}s dc-losses and/or the resistive dc loads. %
The \ac{ilc}s modulation signals are captured by $m(\delta)  \triangleq (m_1(\delta_1),\ldots,m_n(\delta_n)) \in \R{ 2n \times n }$ with $ m_j(\delta_j)=\mu_j r(\delta_j) \otimes e_j \in\R{2n}$ where $e_j$ denotes the $j$-th orthonormal basis of $\R{n}$, $r(\delta_j)  \triangleq  ( \cos(\delta_j) , \sin(\delta_j) )$, and $\mu_j\in\R{}_{[0,1/2]}$ denotes the $j$-th modulation signal magnitude. Last, $i  \triangleq  ( i_1 , \ldots , i_n ) \in \R{2n}$ is the vector of the currents flowing through the \acp{ilc} output filters.

Furthermore, $L  \triangleq  \di{ \{ L_j \otimes I_2 \}^{n}_{j=1} } \in \R{ 2n \times 2n }$ denotes the augmented diagonal matrix of \acp{ilc} filter inductances, $R  \triangleq  \di{ \{ R_j \otimes I_2 \}_{ j = 1 } ^ n} \in \R{2n \times 2n}$ denotes the resistance matrix associated with the filter impedance, and $\ubar{L}  \triangleq  \di{ \{ L_j \}_{ j = 1 } ^ n} \in \R{n \times n}$ is the reduced version of $L$, $C  \triangleq  \di{ \{ C_j \otimes I_2 \}^{n}_{j=1} } \in \R{ 2n \times 2n }$ denotes the augmented diagonal matrix of filter capacitance, $G  \triangleq  \di{ \{ G_j \otimes I_2 \}_{ j = 1 } ^ n} \in \R{2n \times 2n}$ is the filter conductance, and  $\ubar{C}  \triangleq  \di{ \{ C_j \}_{ j = 1 } ^ n} \in \R{n \times n}$ is the reduced version of $C$. 

Note that $G$, $C$, and $\ubar{C}$ also incorporate the shunt conductances of the ac lines. Similarly, $G_\tx{dc}$ and $C_\tx{dc}$ incorporate the shunt conductances and capacitances of the dc lines. Last, $G$ and $G_\tx{dc}$ can as well represent the local resistive loads that are collocated with the \acp{ilc} ac and dc-sides.
\subsection{Dynamic DC interconnections}
We model the dc lines with RL dynamics that \cite{zonetti2015control}, i.e.,
\begin{equation}\label{eq:dc lines}
{ \dot{i}_\tx{dc,n} } = L_\tx{dc}^{-1} \left( - B^\top v_\tx{dc} - R_\tx{dc} i_\tx{dc,n} \right),
\end{equation}
where $L_\tx{dc}  \triangleq  \di{ \{ L_{\tx{dc},j} \}^m_{j=1} } \in \R{ m \times m }$ and $R_\tx{dc}  \triangleq  \di{ \{ R_{\tx{dc},j} \}^m_{j=1} } \in \R{ m \times m }$ respectively denote the diagonal inductance and resistance matrices associated with the dc lines. Note that we do not make any assumption on the sparsity of the underlying graph\footnote{If the dc interconnections includes disconnected subgraphs, the stability result in Section \ref{sec:stability} holds for the individual hybrid ac/dc subsystems.} associated with the dc interconnections.  

One can augment the dc network with constant current sources that model distributed wind generations (assumed to be stiffly controlled as in \eqref{eq:ILC2} on the timescales of interest). Such augmentation only shifts the equilibria and that is why the analysis in the next section suggests that the overall stability is not affected by location or control of the dc sources.   
\section{Hybrid Angle Control and Stability Analysis}\label{sec:stability}
In this section, we equip the \acp{ilc} with grid-forming \ac{hac} strategy, formulate the closed-loop dynamics, and present the stability analysis of the hybrid ac/dc grids under HAC.
\subsection{Hybrid angle control for interlinking converters}\label{subsec:hac}
We define the frequency of the \acp{ilc} in \eqref{eq:ILC1} according to the multi-variable grid-forming \ac{hac} \cite{HACCDC,HACTAC,HACPSCC}, i.e.,
\begin{equation}\label{eq:HAC}
\omega_\tx{c}\triangleq\omega_\tx{r} + \eta(v_\tx{dc}-v_\tx{dc,r}) - \gamma \mathbf{sin} \left( \dfrac{\delta - \delta_\tx{r}}{2}\right),
\end{equation}
where $\eta \triangleq \di{ \{ \eta_j \}^n_{j=1} } \in \R{ n \times n }$ and $\gamma  \triangleq  \di{ \{ \gamma_j \}^n_{j=1} } \in \R{ n \times n }$ respectively denote the diagonal matrix of the dc and ac gains associated with \ac{hac}. %
Further, for any $x \in \R{n}$, $\mathbf{sin}( x )  \triangleq  ( \sin(x_1) , \ldots , \sin(x_n) ) $. %
Last, $ \delta  \triangleq  \theta_\tx{c} - \theta_\tx{g} $ denotes the vector of relative \ac{ilc}-\ac{coi} angles and $\delta_\tx{r}  \triangleq   ( \delta_{\tx{r},1} , \ldots , \delta_{\tx{r},n} ) \in \mathbb{S}^n$ collects the reference relative angles. %
In what follows, we consider that all angular quantities evolve on the boundary of a Möbius strip, i.e., $\mathbb{M}$ \cite{HACTAC}, hence, $\delta\in\mathbb{M}^n$; see Figure \ref{fig:KB} fo a geometrical representation of $\mathbb{M}^2$.    

From a philosophical viewpoint, \ac{hac} \eqref{eq:HAC} resembles the emerging hybrid control laws e.g., see \cite{chen2021generalized,chen2022augmentation,SG21,gross2021dualport,gao2020grid}. In contrast to the classic techniques, e.g.,  \cite{WL20,ahrabi2020hybrid,zonetti2015control,peyghami2017autonomous}, the hybrid strategies unify the dc and ac measurements into a single controller. In particular, \ac{hac} inherently encodes trade-off between the dc voltages (that relate to the dc energies) and the ac angles (that relate to the ac power flows) deviations from the respective references. We refer to \cite{HACCDC,HACTAC,HACPSCC} for further details on the \ac{hac} behavior and properties and \cite{SG21} for elaborations on the hybrid control architectures.     
\begin{figure}[t!]
	\centering
	\includegraphics[width=0.65\columnwidth]{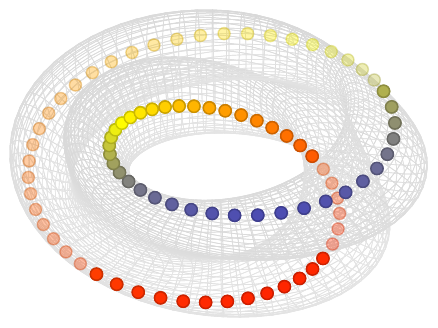}
	\caption{The figure-$8$ immersion of the Klein bottle \cite{sequin2013}. The Klein bottle can be decomposed into two M\"{o}bius strips (e.g., the ones above and below the colored path) with identical compact boundaries representing $\mathbb{M}$ \cite{HACTAC}. The merged boundaries of the underlying M\"{o}bius strips (e.g., the colored closed curve) represents the angle manifold $\mathbb{M}^2=]-2\pi,2\pi[\times]-2\pi,2\pi[$ where $-2\pi\equiv2\pi$.}\label{fig:KB}
\end{figure}
\subsection{Closed-loop analysis}\label{subsec:equilibria}

In order to combine the models introduced in Section \ref{sec:modelling}, we first define the aggregated electrical torque and the voltage associated with \eqref{eqs:coi}. Similar to the modeling approach in \cite{HACTAC,BSEO17} with define the $j$-th stiff \ac{coi} voltage (that resembles the synchronous generator electromagnetic force) as
\begin{equation}
v_{\tx{g,abc},j}  \triangleq b_j \omega_{\tx{g},j} \big( \sin \theta_{\tx{g},j}  , \sin (\theta_{\tx{g},j} - \tfrac{2\pi}{3})  , \sin (\theta_{\tx{g},j} - \tfrac{4\pi}{3}) \big)
\label{eq:coi voltage}
\end{equation} 
where $b_j \in \R{}_{>0}$ is a constant. Note that we can alternatively simplify the frequency-dependent magnitude in \eqref{eq:coi voltage} to a constant reference $v_{\tx{r},j}$. The implicit assumption in \eqref{eq:coi voltage} is that $b_j  \triangleq  v_{\tx{r},j} / \omega_{\tx{g},j}^\star$ realizes the desired magnitude at the equilibrium frequency $\omega_{\tx{g},j}^\star$ for the $j$-th ac grid. Subsequently, the $j$-th electrical torque in \eqref{eq:coi2} is defined by \cite{HACTAC,BSEO17} 
\begin{equation}
T_{\tx{e},j}  \triangleq  \omega_{\tx{g},j}^{-1} v_{\tx{g,abc},j}^\top i_{\tx{g,abc},j}.
\label{eq:electric torque}
\end{equation}
Finally, combining \eqref{eqs:coi}-\eqref{eq:electric torque} yields the overall dynamics, i.e.,
\begin{equation}\label{eqs:cl sys}
	\dot{x}=K^{-1}f(x),
\end{equation}
where
\begin{align*}
	{x} &\triangleq (\delta,i_\tx{dc,n},i_\tx{dc,g},v_\tx{dc},i,v,i_\tx{g},\omega_\tx{g},T_\tx{m}),
	\\
	K &\triangleq \di{I_n,L_\tx{dc},\tau_\tx{dc},C_\tx{dc},L,C,L_\tx{g},J,\tau_\tx{g}},
	\\
	f(x) &\triangleq \begin{pmatrix}
		\omega_\tx{r} + \eta(v_\tx{dc}-v_\tx{dc,r}) - \gamma \mathbf{sin} \left( ( \delta - \delta_\tx{r} ) / 2 \right) - \omega_\tx{g}
		\\
		- B^\top v_\tx{dc} - R_\tx{dc} i_\tx{dc,n}
		\\
		i_\tx{dc,r} - \kappa_\tx{dc} ( v_\tx{dc} - v_\tx{dc,r} ) - i_\tx{dc,g}
		\\
		i_\tx{dc,g} + B i_\tx{dc,n} - G_\tx{dc} v_\tx{dc} - m( \delta )^\top i	
		\\
		m( \delta ) v_\tx{dc} - ( R - \ubar{L} \omega_\tx{g} \otimes J_2 ) i - v
		\\
		i - ( G - \ubar{C} \omega_\tx{g} \otimes J_2 ) v - i_\tx{g}
		\\
		v - ( R_{\tx{g}} - \ubar{L}_{\tx{g}} \omega_\tx{g} \otimes J_2 ) i_\tx{g} - \psi \omega_\tx{g}
		\\
		T_\tx{m} - D_\tx{f} \omega_\tx{g} - D_\tx{d} (\omega_\tx{g} - \omega_\tx{r})  + \psi^\top i_\tx{g}
		\\
		T_\tx{r} - \kappa_\tx{g} ( \omega_\tx{g} - \omega_\tx{r} ) - T_\tx{m}
	\end{pmatrix},
\end{align*}
%
in which \eqref{eq:coi voltage} and \eqref{eq:electric torque} are transformed to the dq-frames aligned with $\theta_{\tx{g},j}$ and written in terms of $\psi\triangleq( \psi_1 , \ldots , \psi_n ) \in \R{ 2n \times n }$ with $\psi_j  \triangleq  b_j r(0) \otimes e_j \in\R{2n}$. We partition the state vector as $x \triangleq (\delta,y)\in\mathbb{X} \triangleq \mathbb{M}^n\times\R{10n+m}$ where $y \triangleq (i_\tx{dc,n},i_\tx{dc,g},v_\tx{dc},i,v,i_\tx{g},\omega_\tx{g},T_\tx{m})$ and remark that $f(x)$ is smooth in $\mathbb{X}$. Last, for notational convenience we define $D \triangleq D_\tx{f} + D_\tx{d}$. 
\begin{assumption}[Frequency and dc voltage regulation]	\label{AS:existence}
Assume that the stationary frequency $\omega_\tx{g}^\star$ and dc voltage $v_\tx{dc}^\star$ of \eqref{eqs:cl sys} coincide with the respective references $\omega_\tx{r}$ and $v_\tx{dc,r}$. 
\end{assumption}
Assumption \ref{AS:existence} implies requirements for frequency and dc voltage balancing across the ac/dc grids. This is met by an appropriate choice of reference-parameter pairs $(T_\tx{r},\kappa_\tx{g})$ and $(i_\tx{dc,r},\kappa_\tx{dc})$ in \eqref{eq:coi3} and \eqref{eq:ILC2}, respectively \cite{HACTAC}. Note that considering secondary integral-type controllers in \eqref{eq:coi3} and \eqref{eq:ILC2} also ensures that Assumption \ref{AS:existence} holds, but, the integral control hinders the frequency and dc voltage droop mechanisms that are crucial for load-sharing \cite{WL20}. Thus, the blend of consistent references and adequately tuned proportional controllers is recommended for verification of Assumption \ref{AS:existence}. 
\begin{theorem}[Existence and uniqueness]
Under Assumption \ref{AS:existence}, the closed-loop dynamics \eqref{eqs:cl sys} admits a unique equilibrium set that is described by
\begin{equation}\label{eq:equilibrium set}
	\Omega^\star \triangleq \left\{(\delta^\star,y^\star)\big|\delta^\star_j\in\{\delta_{\tx{r},j},\delta_{\tx{r},j}+2\pi\},\forall j\in\mathcal{N}_\tx{ac}\right\},
\end{equation}  
where $y^\star$ is unique \ac{wrt} $(\delta_\tx{r},v_\tx{dc,r},\omega_\tx{g,r})$ and $\Omega^\star$ only contains disjoint points that only differ in their angles.   
\end{theorem}
\begin{proof}
Setting the \ac{rhs} of \eqref{eqs:cl sys} of to zero, by Assumption \eqref{AS:existence}, angle dynamics \eqref{eqs:cl sys} at the equilibrium, i.e.,
\begin{equation}
	\omega_\tx{r}+\eta(v_\tx{dc}^\star-v_\tx{dc,r})-\gamma\mathbf{sin}\left({(\delta^\star-\delta_\tx{r})}/{2}\right)-\omega^\star_\tx{g}=0,
\end{equation}	
reduces to $\mathbf{sin}\left({(\delta^\star-\delta_\tx{r})}/{2}\right)=0$. This implies that the elements of the angle equilibrium $\delta^\star$, i.e., $\delta^\star_j\in\{\delta_{\tx{r},j},\delta_{\tx{r},j}+2\pi\}$ for all $j\in\mathcal{N}_\tx{ac}$. Further, Assumption \ref{AS:existence} implies the existence of dc voltage and frequency equilibria, thus, their respective dynamics in \eqref{eqs:cl sys} vanish at the equilibrium. Hence, $i_\tx{dc,n}^\star=-R^{-1}B^\top v_\tx{dc,r}$, $i_\tx{dc,g}^\star=i_\tx{dc,r}$, and $T^\star_\tx{r}=T_\tx{r}$ that follow from the dc edge, dc generation, and torque dynamics in \eqref{eqs:cl sys} at the equilibrium, respectively. Next, the ILCs' filter and transmission dynamics can be written as $F\ubar{y}^\star=h$ where $F\triangleq$
\begin{align*}
			\setlength{\arraycolsep}{-0.5pt}
	\left(\begin{array}{ccc}
		- ( R - \ubar{L} \omega_\tx{g} \otimes J_2 )  & -I_{2n} & 0
		\\
		I_{2n} & - ( G - \ubar{C} \omega_\tx{g} \otimes J_2 ) & - I_{2n}
		\\
		0 & I_{2n} & - ( R_{\tx{g}} - \ubar{L}_{\tx{g}} \omega_\tx{g} \otimes J_2 )
		\end{array}\right), 
\end{align*}
$\ubar{y}^\star\triangleq\left(i^\star,v^\star,i^\star_\tx{g}\right)$, and $h\triangleq\left(-m( \delta^\star ) v_\tx{dc}^\star,0,\psi \omega_\tx{g}^\star\right)$.
Note that as in \cite[Theorem 1]{HACTAC}, the symmetric part of $F$, i.e., $(1/2)\left(F+F^\top\right)\prec0$ that means $F$ is invertible and $\ubar{y}^\star$ is unique. Thus, the $y^\star$ in \eqref{eq:equilibrium set} is uniquely given by $y^\star= (i_\tx{dc,n}^\star,i_\tx{dc,g}^\star,v_\tx{dc}^\star,i^\star,v^\star,i_\tx{g}^\star,\omega_\tx{g}^\star,T_\tx{m}^\star)$ that completes the proof.
\end{proof}
	 In contrast to the single converter dynamics in \cite{HACCDC,HACTAC,BSEO17} that only admit two equilibria, \eqref{eqs:cl sys} admits more disjoint equilibria. 
	Among all the points in $\Omega^\star$, $x^\star_\tx{s} \triangleq (\delta_\tx{r},y^\star)$ has a particularly interesting stability nature (more on this later). Last, for notational convenience we define $\Omega^\star_\tx{u} \triangleq \Omega^\star - x^\star_\tx{s}$.
\begin{theorem}[Decentralized certificates for global attractivity]\label{TH:1}
The equilibria of system \eqref{eqs:cl sys} as in \eqref{eq:equilibrium set} are globally attractive if the following decentralized conditions hold for all $j\in\mathcal{N}_\tx{ac}$:
\begin{equation}\label{eqs:stability conditions}
D_j>D_{\tx{min},j}\quad\tx{and}\quad \gamma_j>\gamma_{\tx{min},j},
\end{equation}
where the critical COI damping, i.e., $D_{\tx{min},j}$ is defined by  
	\begin{equation*}
		\dfrac{\left(L_j\n{i^\star_j}\right)^2}{R_j} + \dfrac{\left(C_j\n{v^\star_j}\right)^2}{G_j} + \dfrac{\left(L_{\tx{g},j}\n{i^\star_{\tx{g},j}}\right)^2}{R_{\tx{g},j}},  
	\end{equation*}
and the critical ILC angle damping, i.e., $\gamma_{\tx{min},j}$ is defined by
\begin{equation*}
\dfrac{\eta_j\left(1+\left(\mu_j\n{i^\star_j}\right)^2\right)}{G_{\tx{dc},j}} + \dfrac{\eta_j\left(\mu_j\n{v^\star_{\tx{dc},j}}\right)^2}{R_j} + \dfrac{1}{2\left(D_j-D_{\tx{min},j}\right)}. 
\end{equation*}
\end{theorem}
\begin{proof}
Define the linear error coordinates $\hat{x}\in\mathbb{X}$ \ac{wrt} $x^\star_\tx{s}=(\delta_\tx{r},y^\star)$ (as defined in Subsection \eqref{subsec:equilibria}), i.e.,
\begin{align}
	\hat{x} \triangleq (\hat{\delta},\hat{y})& \triangleq (\delta-\delta_\tx{r},i_\tx{dc,n}-i_\tx{dc,n}^\star,i_\tx{dc,g}-i_\tx{dc,g}^\star,v_\tx{dc}-v_\tx{dc}^\star,\nonumber
	\nonumber
	\\
	&i-i^\star,v-v^\star,i_\tx{g}-i_\tx{g}^\star,\omega_\tx{g}-\omega_\tx{g}^\star,T_\tx{m}-T_\tx{m}^\star).\label{eq:error coordinates 1}
\end{align}
Subsequently, the translation of the closed-loop dynamics \eqref{eqs:cl sys} to the coordinates \eqref{eq:error coordinates 1} results in the error dynamics, i.e., 
\begin{equation}\label{eqs:error dynamics}
	\dot{\hat{x}}=K^{-1}\hat{f}(\hat{x}), 
\end{equation}
where $\hat{f}(\hat{x})\triangleq f(\hat{x}+x^\star_\tx{s})$ and is given by 
\begin{equation*}
	\renewcommand\arraystretch{1.1}
	\begin{pmatrix}
		\eta \hat{v}_\tx{dc} - \gamma \mathbf{sin} \big( \hat{\delta} / 2 \big) - \hat{\omega}_\tx{g}
		\\
		- B^\top \hat{v}_\tx{dc} - R_\tx{dc} \hat{i}_\tx{dc,n} 
		\\
		- \kappa_\tx{dc} \hat{v}_\tx{dc} - \hat{i}_\tx{dc,g} 
		\\
		\hat{i}_\tx{dc,g} + B \hat{i}_\tx{dc,n} - G_\tx{dc} \hat{v}_\tx{dc} - E(\delta)^\top i^\star - m( \delta )^\top \hat{i}
		\\
		m( \delta ) \hat{v}_\tx{dc} + E(\delta) v_\tx{dc}^\star - R \hat{i} - \ubar{L} \omega_\tx{g} \otimes J_2 \hat{i} -\ubar{L} \hat{\omega}_\tx{g} \otimes J_2 i^\star - \hat{v}
		\\
		\hat{i} - G \hat{v} - \ubar{C} \omega_\tx{g} \otimes J_2 \hat{v} -\ubar{C} \hat{\omega}_\tx{g} \otimes J_2 v^\star - \hat{i}_\tx{g}
		\\
		\hat{v} - R_\tx{g} \hat{i}_\tx{g} - \ubar{L}_\tx{g} \omega_\tx{g} \otimes J_2 \hat{i}_\tx{g} -\ubar{L}_\tx{g} \hat{\omega}_\tx{g} \otimes J_2 i^\star_\tx{g}  -\psi \hat{\omega}_\tx{g}
		\\
		\hat{T}_\tx{m} - D \hat{\omega}_\tx{g} + \psi^\top \hat{i}_\tx{g}
		\\
		- \kappa_\tx{g} \hat{\omega}_\tx{g} - \hat{T}_\tx{m} 
	\end{pmatrix},
\end{equation*}
where we exploited the fact that $f(x^\star_\tx{s})=0$ and $E(\delta) \triangleq m( \delta ) - m( \delta^\star )$ denotes the vector of the trigonometric modulation errors. Consider the LaSalle function candidate:
\begin{equation}\label{eq:Lyapunov function}
	\mathcal{V}(\hat{x}) \triangleq \mc{S}(\hat{ \delta }) + \mc{H}(\hat{ y }) =  {2\sum_{j\in\mathcal{N}_\tx{ac}}\lambda_j\left(1-\cos\dfrac{\hat{\delta_j}}{2}\right)}+ {\dfrac{1}{2}\left(\hat{y}^\top P \hat{y}\right)},
\end{equation}
where for all $j\in\mathcal{N}_\tx{ac}$, $\lambda_j \in \R{}_{>0}$ is a free parameter and $P =: \di{ L_\tx{dc} , \tau_\tx{dc} \kappa_\tx{dc}^{-1} , C , L , L_\tx{g} , J , \tau_\tx{g} \kappa_\tx{g}^{-1} }\succ 0$ (with the well-defined model and control parameters). Note that $\mathcal{V}(\hat{x})>0$ for all $\hat{x}\neq 0$ (modulo $4\pi$). For notational convenience we collect all $\lambda_j$ in $\lambda \triangleq \di{\{\lambda_j\}^{n}_{j=1}}$. We evaluate the time derivative of $\mathcal{V}(\hat{x})$ along the solutions of \eqref{eqs:error dynamics}, that is, 
\begin{align}\label{eq:dV/dt 2}
	\dot{\mathcal{V}}(\hat{x})=~&\mathbf{sin} ( \hat{\delta} / 2 )^\top \left(\lambda \eta \hat{v}_\tx{dc} - \lambda \gamma~ \mathbf{sin} ( \hat{\delta} / 2 ) -  \lambda \hat{\omega}_\tx{g}\right) -
	\nonumber
	\\
	&\hat{i}^\top R \hat{i} - \hat{v}^\top G \hat{v} - \hat{i}_\tx{g}^\top R_\tx{g} \hat{i}_\tx{g} - \hat{\omega}_\tx{g}^\top D \hat{\omega}_\tx{g} - \hat{T}_\tx{m}^{\top} \kappa_\tx{g}^{-1} \hat{T}_\tx{m}-
	\nonumber
	\\
	&
	\hat{i}^\top \ubar{L} \hat{\omega}_\tx{g} \otimes J_2 i^\star   - \hat{v}^\top \ubar{C} \hat{\omega}_\tx{g} \otimes J_2 v^\star   - \hat{i}_\tx{g}^\top \ubar{L}_\tx{g} \hat{\omega}_\tx{g} \otimes J_2 i^\star_\tx{g}-
	\nonumber
	\\
	&\hat{i}^\top_\tx{dc,n} R_\tx{dc} \hat{i}_\tx{dc,n} - \hat{i}_\tx{dc,g}^\top \kappa_\tx{dc}^{-1} \hat{i}_\tx{dc,g} - \hat{v}_\tx{dc}^\top G_\tx{dc} \hat{v}_\tx{dc} - 
	\nonumber
	\\
	&\hat{v}_\tx{dc}^\top E(\delta)^\top i^\star  + \hat{i}^\top E(\delta) v_\tx{dc}^\star.
\end{align}
We derive decoupled bounds on the $E(\delta)$-cross-terms, e.g., 
\begin{equation*}
	- \hat{v}_\tx{dc}^\top E(\delta)^\top i^\star=	
	- \sum_{j\in\mathcal{N}_\tx{ac}} \mu_j\hat{v}_{\tx{dc},j} \left( E\left(\delta\right) \otimes e_j \right)^\top i^\star_j,
\end{equation*}
that (by analogous bounding schemes as in \cite[Lemma 1]{HACTAC}) is upper-bounded by
\begin{equation*}
	\sum_{j\in\mathcal{N}_\tx{ac}} \left(\epsilon_{1,j}\mu_j \n{i^\star_j}\right)^2 \hat{v}_{\tx{dc},j}^2 + \sum_{j\in\mathcal{N}_\tx{ac}} \left({\epsilon_{1,j}^{-2}}/{4}\right) \n{E(\delta) \otimes e_j}^2,
\end{equation*}
where $\epsilon_1  \triangleq  (\epsilon_{1,1},\ldots,\epsilon_{1,n})\in\R{n}_{>0}$ is a constant vector. Applying the trigonometric angle difference and half-angle identities \cite[Lemma 1]{HACTAC} yields that 
\begin{equation*}
	\n{E(\delta) \otimes e_j}^2=2(1-\cos(\hat{ \delta_j }))=4\sin(\hat{ \delta_j }/2).
\end{equation*}
Thus, the bound on $- \hat{v}_\tx{dc}^\top E(\delta)^\top i^\star$ is alternatively expressed by
\begin{equation*}
	\sum_{j\in\mathcal{N}_\tx{ac}} \left(\epsilon_{1,j}\mu_j \n{i^\star_j}\right)^2 \hat{v}_{\tx{dc},j}^2 +\sum_{j\in\mathcal{N}_\tx{ac}} {\epsilon_{1,j}}^{-2}\sin^2(\hat{ \delta_j }/2),
\end{equation*}
that takes the compact form as in \eqref{eq:upper bound 1}. Similarly, the bound on the other $E(\delta)$ term in \eqref{eq:dV/dt 2} is obtained as in \eqref{eq:upper bound 2} where $\epsilon_2  \triangleq  (\epsilon_{2,1},\ldots,\epsilon_{2,n})\in\R{n}_{>0}$ is a constant vector. 

Next, we derive the upper-bounds on the cross terms in \eqref{eq:dV/dt 2} that depend on $i^\star$, $v^\star$, and $i_\tx{g}^\star$. These terms arise due to the time-varying angular frequency of the dq-coordinates, i.e., $\omega_\tx{g}$. To this end, we employ the same bounding technique as in \eqref{eq:upper bound 1} and \eqref{eq:upper bound 2}, that results in the bounds \eqref{eqs:upper bound 3}-\eqref{eqs:upper bound 5} where $\epsilon_3  \triangleq  (\epsilon_{3,1},\ldots,\epsilon_{3,n})\in\R{n}_{>0}$, $\epsilon_4  \triangleq  (\epsilon_{4,1},\ldots,\epsilon_{4,n})\in\R{n}_{>0}$, and $\epsilon_5  \triangleq  (\epsilon_{5,1},\ldots,\epsilon_{5,n})\in\R{n}_{>0}$  are all constant vectors. Taking into account the bounds \eqref{eq:upper bound 1}-\eqref{eqs:upper bound 5}, we evaluate an upper-bound on the \ac{rhs} of $\dot{\mathcal{V}}(\hat{x})$ in \eqref{eq:dV/dt 2}, that is,
\begin{align}\label{eq:dV/dt 3}
	& \mathbf{sin} ( \hat{\delta} / 2 )^\top \left(\lambda \eta \hat{v}_\tx{dc} 
	-  \left(\lambda \gamma-\varphi_2-\varphi_4\right)~ \mathbf{sin} ( \hat{\delta} / 2 ) 
	-  \lambda \hat{\omega}_\tx{g}\right)-	
	\nonumber
	\\
	&\hat{i}^\top_\tx{dc,n} R_\tx{dc} \hat{i}_\tx{dc,n} - \hat{i}_\tx{dc,g}^\top \kappa_\tx{dc}^{-1} \hat{i}_\tx{dc,g}- \hat{v}_\tx{dc}^\top \left(G_\tx{dc}-\varphi_1\right) \hat{v}_\tx{dc}- 
	\nonumber
	\\
	& \hat{i}^\top \left(R-\varphi_3-\varphi_5\right) \hat{i} 
	- \hat{v}^\top \left(G-\varphi_7\right) \hat{v} - \hat{i}_\tx{g}^\top \left(R_\tx{g}-\varphi_9\right) \hat{i}_\tx{g} -
	\nonumber
	\\
	&\hat{\omega}_\tx{g}^\top \left(D -\varphi_6-\varphi_8-\varphi_{10}\right) \hat{\omega}_\tx{g} - \hat{T}_\tx{m}^{\top} \kappa_\tx{g}^{-1} \hat{T}_\tx{m}.
\end{align}
Consider the alternative coordinates partioning $\hat{\ubar{x}} \triangleq \left(\hat{\ubar{x}}_{1},\hat{\ubar{x}}_2\right)$ that, compared to $\hat{x}$ in \eqref{eq:error coordinates 1}, replaces $\hat{\delta}$ with its nonlinear counterpart $\mathbf{sin}(\hat{ \delta })/2$ and reshuffles the elements of $\hat{x}$  as 
\begin{align*}
	\hat{\ubar{x}}_1& \triangleq \left(\sin\big(\hat{ \delta }_1/2\big),\hat{v}_{\tx{dc},1},\hat{\omega}_1,\ldots, \sin\big(\hat{ \delta_n }/2\big),\hat{v}_{\tx{dc},n},\hat{\omega}_n\right),  
	\\
	\hat{\ubar{x}}_2& \triangleq  \left( \hat{i}_\tx{dc,n} , \hat{i}_\tx{dc,g} ,  \hat{i} , \hat{v} , \hat{i}_\tx{g} , \hat{T}_\tx{m} \right).
\end{align*}
Hence, the \ac{rhs} of \eqref{eq:dV/dt 3} takes a quadratic form in $\hat{\ubar{x}}$ i.e.,       
\begin{equation}\label{eq:quadratic bound 1}
	{\dot{\mathcal{V}}(\hat{x})}	\leq 
	- \hat{\ubar{x}}^\top 	Q	 \hat{\ubar{x}}=-\left(\hat{\ubar{x}}_1^\top 	Q_{11}	 \hat{\ubar{x}}_1 + \hat{\ubar{x}}_2^\top 	Q_{22}	 \hat{\ubar{x}}_2\right)
\end{equation}
where ${Q}_{11} \triangleq \di{\left\{{Q}_{11,j}\right\}_{j=1}^{n}}$ with ${Q}_{11,j}$ as in \eqref{eq:Q_11} and $Q_{22}  \triangleq \di{ R_\tx{dc} , \kappa_\tx{dc}^{-1} , R-\varphi_3-\varphi_5 , G-\varphi_7 , R_\tx{g}-\varphi_9 , \kappa_\tx{g}^{-1} }$.

Now, let us assign the free parameters as $\lambda_j=2/\eta_j$, $\epsilon_{1,j}=\sqrt{G_{\tx{dc},j}}/\left(\sqrt{2}\mu_j\n{i^\star_j}\right)$, $\epsilon_{2,j}=\sqrt{R_j/2}$, $\epsilon_{3,j}=\sqrt{R_j}/2$, $\epsilon_{4,j}=\sqrt{G_j}/2$, and $\epsilon_{5,j}=\sqrt{R_{\tx{g},j}}/2$ for all $j\in\mathcal{N}_\tx{ac}$. This set of parameters directly implies the positive definiteness of $Q_{22}$. Next, standard Schur complement analysis yields that ${Q}_{11}$ is positive definite if and only if \eqref{eqs:stability conditions} is satisfied.
	
Since $\dot{\mc{V}}(\hat{x})\leq0$ for any $\hat{x}(0)\in\mathbb{X}$, then the $c$-sublevel sets of $\mc{V}(\hat{x})$, i.e., $\mathscr{L}_c \triangleq \big\{\hat{x}\in\mathbb{X}:\mc{V}(\hat{x})\leq c\big\}$ with $c \triangleq \mc{V}\left(\hat{x}(0)\right)$, is a forward invariant and compact due to the boundedness of $\hat{ \delta }$ in $\mathbb{M}^n$ (that is the union of $n$ compact Möbius strip boundaries) and the radial unboundedness of $\mathcal{H}(\hat{y})$. Hence, by invoking the LaSalle's invariance principle \cite{khalil_nonlinear_2002}, the solutions of \eqref{eqs:error dynamics} globally converge to the largest invariant set $\mc{M}\subset\Omega \triangleq \big\{\hat{x}\in\mathbb{X}:\dot{\mc{V}}(\hat{x})=0\big\}$. Under the conditions  \eqref{eqs:stability conditions}, ${Q}\succ 0$ in \eqref{eq:quadratic bound 1}. Thus, $\dot{\mc{V}}(\hat{x})=0$ iff $\ubar{\hat{x}}=0$. Finally, $\ubar{\hat{x}}=0$ characterizes a set that is identical to $\Omega^\star$ in \eqref{eq:equilibrium set}, i.e.,  $\Omega=\Omega^\star$.     
\end{proof}
\begin{figure*}[t!]
	\begin{equation}\label{eq:upper bound 1}
		- \hat{v}_\tx{dc}^\top E(\delta)^\top i^\star \leq  \hat{v}_\tx{dc}^\top \varphi_1 \hat{v}_\tx{dc} + \mathbf{sin}(\hat{ \delta }/2)^\top \varphi_2 \mathbf{sin}(\hat{ \delta }/2),
		~
		\varphi_1 \triangleq \di{\left\{\left( { \epsilon_{1,j}\mu_j } \n{i^\star_j}\right)^2\right\}_{j=1}^{n}},~\varphi_2 \triangleq { \di{\left\{\epsilon_{1,j}^{-2}\right\}_{j=1}^{n}} }
	\end{equation}
	\vspace{-2.3mm}
	\begin{equation}\label{eq:upper bound 2}
		\hat{i}^\top E(\delta) v_\tx{dc}^\star \leq \hat{i}^\top \varphi \hat{i} + \mathbf{sin}(\hat{ \delta }/2)^\top \varphi_4 \mathbf{sin}(\hat{ \delta }/2),
		~
		\varphi_3 \triangleq  \di{\left\{\epsilon_{2,j}^2\otimes I_2\right\}_{j=1}^{n}},~\varphi_4 \triangleq \di{\left\{ \left({\mu_j v_{\tx{dc},j}^\star}/{\epsilon_{2,j}}\right)^2 \right\}_{j=1}^{n}}
	\end{equation}
	\vspace{-2.3mm}
	\begin{equation}\label{eqs:upper bound 3}
		- \hat{i}^\top \ubar{L} \hat{\omega}_\tx{g} \otimes J_2 i^\star \leq \hat{i}^\top \varphi_5 \hat{i} + \hat{\omega}_\tx{g}^\top \varphi_6 \hat{\omega}_\tx{g},
		~
		\varphi_5 \triangleq \di{\left\{\epsilon_{3,j}^2\otimes I_2\right\}_{j=1}^{n}}, ~\varphi_6 \triangleq \di{ \left\{ \left( { L_j\n{i_j^\star} }/ { 2 \epsilon_{3,j} }  \right)^2 \right\}_{j=1}^{n}}
	\end{equation}
	\vspace{-2.3mm}
	\begin{equation}\label{eqs:upper bound 4}
		- \hat{v}^\top \ubar{C} \hat{\omega}_\tx{g} \otimes J_2 v^\star \leq \hat{v}^\top \varphi_7 \hat{v} + \hat{\omega}_\tx{g}^\top \varphi_8 \hat{\omega}_\tx{g},
		~
		\varphi_7 \triangleq \di{\left\{\epsilon_{4,j}^2\otimes I_2\right\}_{j=1}^{n}},~\varphi_8 \triangleq \di{ \left\{ \left(  C_j \n{v_j^\star \right)^2} / { 2 \epsilon_{4,j} }  \right\}_{j=1}^{n}}
	\end{equation}
	\vspace{-2.3mm}
	\begin{equation}\label{eqs:upper bound 5}
		- \hat{i}_\tx{g}^\top \ubar{L}_\tx{g} \hat{\omega}_\tx{g} \otimes J_2 i^\star_\tx{g}\leq \hat{i}_\tx{g}^\top \varphi_9 \hat{i}_\tx{g} + \hat{\omega}_\tx{g}^\top \varphi_{10} \hat{\omega}_\tx{g},
		~
		\varphi_9 \triangleq \di{\left\{\epsilon_{5,j}^2\otimes I_2\right\}_{j=1}^{n}},~\varphi_{10} \triangleq \di{ \left\{ \left( { L_{\tx{g},j}\n{i_{\tx{g},j}^\star} }/ { 2 \epsilon_{5,j} }  \right)^2 \right\}_{j=1}^{n}}
	\end{equation}
	\vspace{-2.3mm}
	\begin{equation}\label{eq:Q_11}
		\begin{pmatrix}
			\lambda_j \gamma_j - \dfrac{1}{\epsilon_{1,j}^2} - \left(\dfrac{\mu_j v_{\tx{dc},j}^\star}{\epsilon_{2,j}}\right)^2	&	- \dfrac{\lambda_j \eta_j}{2}	&	 \dfrac{\lambda_j}{2}		
			\\
			- \dfrac{\lambda_j \eta_j}{2}	&	G_{\tx{dc},j}-\left( { \epsilon_{1,j}\mu_j } \n{i^\star_j}\right)^2		&	0
			\\
			\dfrac{\lambda_j}{2}	&	0	&	D_j-\left( \dfrac { L_j \n{i_j^\star} } { 2 \epsilon_{3,j} } \right)^2 - \left( \dfrac { C_j \n{v_j^\star} } { 2 \epsilon_{4,j} }  \right)^2 - \left( \dfrac { L_{\tx{g},j} \n{i_{\tx{g},j}^\star} } { 2 \epsilon_{5,j} } \right)^2
		\end{pmatrix}
	\end{equation}
	\hrulefill
\vspace{-2.5mm}
\end{figure*}

\begin{corollary}[Local asymptotic stability of $x^\star_\tx{s}$]\label{col:stability}
Consider the closed-loop system \eqref{eqs:cl sys} and the equilibrium point $x^\star_\tx{s}=(\theta_\tx{r},y^\star)$, then $x^\star_\tx{s}$ is locally asymptotically stable if the stability conditions \eqref{eqs:stability conditions} are satisfied for all $j\in\mathcal{N}_\tx{ac}$. 
\end{corollary}
\begin{proof}
 Consider the error coordinates \eqref{eq:error coordinates 1} (that are written \ac{wrt} $x^\star_\tx{s}$) and error dynamics \eqref{eqs:error dynamics}. Note that $\mc{V}(\hat{x})$ vanishes at the origin and $\mc{V}(\hat{x})>0$ otherwise in $\mathbb{X}$. Next, by Theorem \ref{TH:2}, if \eqref{eqs:stability conditions} is satisfied, $\dot{\mc{V}}(\hat{x})<0$ in a sufficiently small open neighborhood of the origin (that excludes any other equilibria in $\Omega^\star$). The existence of such an open neighborhood is guaranteed since all equilibria in $\Omega^\star$ are disjoint. Consider a sufficiently small $\ubar{c}$-sublevel set of $\mc{V}(\hat{x})$ i.e., $\mathscr{L}_{\ubar{c}} \triangleq \left\{\hat{x}\in\mathbb{X}:\mc{V}(\hat{x})\leq \ubar{c},\ubar{c}\in\R{}_{>0}\right\}$ such that it excludes all the equilibria in $\Omega^\star$ except the origin. Note that for sufficiently small $\ubar{c}$, $\dot{\mc{V}}(\hat{x})\leq0$ for all $\hat{x}$ in $\mathscr{L}_{\ubar{c}}$. Thus, $\mathscr{L}_{\ubar{c}}$ is positively invariant along the solutions of \eqref{eqs:error dynamics}. Last, applying the Lyapunov's direct method \cite{khalil_nonlinear_2002} concludes the local asymptotic stability of $x^\star_\tx{s}\in\Omega^\star$.
\end{proof}

\begin{corollary}[Instability and region of attraction of $\Omega^\star_\tx{u}$]\label{col:instability}
Consider the closed-loop system \eqref{eqs:cl sys}, if the conditions \eqref{eqs:stability conditions} are satisfied for all $j\in\mathcal{N}_\tx{ac}$ then all equilibria in $\Omega^\star_\tx{u}$ are unstable with zero-Lebesque-measure region of attractions. 
\end{corollary}
The proof follows by a similar Jacobian and Schur complement analyses as in \cite[Proposition 2]{HACTAC} but is skipped due to the page limitations. Intuitively speaking, $\Omega^\star_\tx{u}$ contains the saddle points of the LaSalle function \eqref{eq:Lyapunov function} since $\mc{H}(\hat{x})$ is globally convex and $\mc{S}(\hat{\delta})$ attains its local maxima on $\Omega^\star_\tx{u}$; see \cite[Remark 6 and Figure 3]{HACTAC}. Based on Theorem \ref{TH:1}, Corollaries \ref{col:stability}, and \ref{col:instability}, Theorem \ref{TH:2} below frames our main result.

\begin{theorem}[Main result: \ac{agas}]\label{TH:2}
The closed-loop system \eqref{eqs:cl sys} is almost globally asymptotically stable with respect to the equilibrium $x^\star_\tx{s}$ if the unified stability/instability conditions \eqref{eqs:stability conditions} are satisfied for all $j\in\mathcal{N}_\tx{ac}$.
\end{theorem}
%
\begin{remark}[Stability conditions]
{First}, conditions \eqref{eqs:stability conditions} are fully decentralized (i.e., they do not require non-local parameters) and confirm that the stability certificate of \ac{hac} for a single converter system is fully scalable; see \cite[Theorem 4]{HACTAC}. 
{Second}, the damping requirement for the \acp{ilc} i.e., $\gamma_j>\gamma_{\tx{min},j}$ does not require large physical damping but is met by an appropriate choice of the control parameters $\gamma_j$ and $\eta_j$. Next, the COIs damping requirements i.e., $D_j>D_{\tx{min},j}$ is a reoccurring theme in related works; see \cite{HACTAC} for details. {Third}, conditions in \eqref{eqs:stability conditions} does not rely on the control of the dc energy sources. Thus, such generation units can be as well distributed within the dc network. {Fourth}, the conditions in \eqref{eqs:stability conditions} suggest that the stability of hybrid ac/dc grids under \ac{hac} does not require timescale separation of the system dynamics. {Last}, a different choice of parameters $\epsilon_j$ for $j=1,\ldots,5$ and $\lambda$ in the proof of Theorem \ref{TH:1} can further relax the conditions \eqref{eqs:stability conditions}. 
\end{remark}

%
%

\begin{remark}[Control implementation and variant]
The reader is referred to \cite{HACCDC,HACTAC,HACPSCC} for discussions on the exact implementation of the \ac{hac} \eqref{eq:HAC}  that are skipped due to the page limitation. Note that under the dc power flow assumption and when reducing the \ac{ilc}s' filters to resistive-inductive elements, \ac{hac} \eqref{eq:HAC} is approximated by 
\begin{equation}\label{eq:HAC2}
	\omega_\tx{c}\approx\omega_\tx{r} + \eta(v_\tx{dc}-v_\tx{dc,r}) - \gamma \mathbf{sin} \left( \dfrac{p - p_\tx{r}}{2}\right),
\end{equation}
where $p$ and $p_\tx{r}$ respectively denote the vector of active power flows between the \ac{ilc}s and \ac{coi}s and and the associated references. The approximation \eqref{eq:HAC2} can be seen as a nonlinear variation of the dual-port grid-forming control in \cite{SG21,gross2021dualport}. 
\end{remark}
\section{Conclusions and Outlook}
In this paper, we presented a dynamical modeling of the hybrid ac/dc grids. Next, we derived fully decentralized conditions for the existence, uniqueness, and global stability of the closed-loop equilibria. Our future work includes: 1) control performance verification via numerical case studies based on real hybrid ac/dc grid models, 2) deriving stability certificates for the systems that incorporate high-fidelity dc energy source models, e.g., wind generators, and 3) stability analysis when decomposing the ac grid models into  distributed generators.   
\section*{Acknowledgment}
Authors thank Catalin Gavriluta and Eduardo Prieto-Araujo for the fruitful discussions and insightful comments.

%
%
%

{\bibliographystyle{IEEEtran}
	\bibliography{IEEEabrv,Ref_LCSS}}
\end{document}